\documentclass[a4paper,autoref]{lipics-v2021}\nolinenumbers

\makeatother
\usepackage{amsmath,amssymb,amsfonts,latexsym}
\usepackage{amsthm}
 
\usepackage{color,graphicx}
\usepackage{url} %,hyperref}
\usepackage{fancybox}
\usepackage{algorithm}
\usepackage[noend]{algpseudocode}
\usepackage{mathtools}
\usepackage{enumerate,xspace}
\usepackage{thm-restate}
\usepackage{siunitx} 
\usepackage{siunitx} 
\usepackage{cleveref}
\usepackage{xcolor}

\hideLIPIcs

\newcommand{\eps}{\varepsilon}
\renewcommand{\epsilon}{\eps}
\newcommand{\etal}{\emph{et al.}\xspace}

\theoremstyle{plain}

\newenvironment{myquote}%
  {\list{}{\leftmargin=4mm\rightmargin=4mm}\item[]}%
  {\endlist}

\renewcommand{\leq}{\leqslant}
\renewcommand{\geq}{\geqslant}

\DeclareMathOperator{\dist}{dist}

\newcommand{\PP}{\mathcal{P}_q}
\newcommand{\LL}{\mathcal{L}_q}

\newcommand{\FF}{\mathbb{F}}
\newcommand{\PG}{\operatorname{PG}}

\renewcommand{\subjclass}[1]{\par\smallskip\noindent\textbf{Mathematics Subject Classification (2020).} #1\par\smallskip}

\counterwithin{lemma}{section}
\counterwithin{theorem}{section}
\counterwithin{proposition}{section}
\counterwithin{claim}{section}
\counterwithin{observation}{section}
\counterwithin{corollary}{section}
\usepackage{amsmath}
\DeclareMathOperator{\girth}{girth}

\title{Sharp quadratic \texorpdfstring{$\chi$}{chi}-binding functions for powers of bipartite graphs}
\author{Arpan Sadhukhan}{Department of Mathematics, Indian Institute of Technology, Dharwad}{ra.a.sadhukhan@iitdh.ac.in}{}{}
\author{Suraj Kumar Sahoo}{Department of Computer Science and Automation, Indian Institute of Science, Bengaluru}{surajks@iisc.ac.in}{}{}

\authorrunning{A.~Sadhukhan, S.~K.~Sahoo} %mandatory. First: Use abbreviated first/middle names. Second (only in severe cases): Use first author plus 'et. al.'
 \Copyright{Arpan Sadhukhan}%mandatory, please use full first names. LIPIcs license is "CC-BY";  http://creativecommons.org/licenses/by/3.0/

\keywords{graph powers, chromatic number, clique number, bipartite graph, chi-boundedness}% mandatory: Please provide 1-5 keywords
\makeatletter
\renewcommand{\subjclassHeading}{%
  \textcolor{lipicsGray}{\fontsize{9}{12}\sffamily\bfseries
    Mathematics Subject Classification (2020)\enskip}%
}

\gdef\@subjclass{05C12, 05C15, 05C50, 05C80}        % <-- put your MSC codes here
\gdef\@ccsdescString{\@subjclass}     % class will print this instead of the red message

\providecommand{\ccsdesc}[2][]{}
\makeatother

\usepackage{textgreek}
\begin{document}

\setcounter{page}{0}
\maketitle
%------------------------------------------------------------------------------------------

%------------------------------------------------------------------------------------------
\begin{abstract}

For every natural number $r\geq 2$, we construct $r^{th}$
powers of bipartite graphs whose chromatic number is quadratic in their clique number, showing that the straightforward quadratic upper bound is best possible. We thereby settle an open problem posed by Chakraborty, Chandran, Jacob and Pillai [J. Graph Theory 112(3) (2026), 235–254] by establishing the sharpness of the quadratic bound for squares of bipartite graphs.
\end{abstract}
%------------------------------------------------------------------------------------------

%\subjclass{05C15, 05C63 }
\section{Introduction}

A proper coloring of a graph $G$ is an assignment of colors to the vertices of $G$ such that adjacent vertices receive distinct colors.  For a graph $G$, the \emph{chromatic number} $\chi(G)$ of $G$ is the minimum number of colors required for a proper coloring of $G$.  In a graph $G$, a clique is a set of pairwise adjacent vertices; the \emph{clique number} $\omega(G)$ is the maximum order of a clique in $G$. 

For a class of graphs $\mathcal{G}$, if there exists a function $f$ such that $\forall G\in \mathcal{G}$, $\chi(G)\leq f(\omega(G))$ then we say that the class of graphs $\mathcal{G}$ is \emph{$\chi$-bounded} and we say that $f$ is the \emph{$\chi$-binding function} of $\mathcal{G}$. Since the chromatic number of a graph is at least its clique number, it is a natural question to ask whether the chromatic number is upper bounded by some function of its clique number. This is precisely the notion of $\chi$-boundedness. Many constructions in the literature demonstrate that the chromatic number of a graph can be arbitrarily larger than its clique number. Therefore, it is clear that the class of all graphs cannot be $\chi$-bounded. But many interesting subclasses do exhibit this property. Scott and Seymour in~\cite{scott2020survey} provide a detailed survey on the $\chi$-boundedness of graphs.

Coloring graph powers (see definition~\ref{def:gpow}), often referred to as distance coloring, has been studied extensively; see the surveys of Kramer and Kramer~\cite{KramerKramer2008} and Cranston~\cite{Cranston2023}. The square $G^2$ is especially well studied. Bounding the chromatic number of the squares of graphs was first studied by Kramer and Kramer~\cite{kramer1969farbungsproblem,kramer1969probleme}. Alon and Mohar~\cite{alon2002chromatic} obtained asymptotically sharp maximum-degree bounds for squares under girth restrictions and also considered extensions to higher powers. More recently, Chakraborty, Chandran, Jacob, and Pillai~\cite{https://doi.org/10.1002/jgt.70014} studied $\chi$-binding functions for squares of bipartite graphs, proved a linear bound for convex bipartite graphs, and asked whether squares of all bipartite graphs admit a subquadratic $\chi$-binding function. Powers of uniform subdivisions were investigated by Anastos, Boyadzhiyska, Rathke, and Rué~\cite{AnastosBoyadzhiyskaRathkeRue2025}. In this paper, we show that the general quadratic upper bound for the $\chi$-binding functions is asymptotically optimal for every fixed power $r\geq 2$ of bipartite graphs.

The following standard observation gives a quadratic upper bound for squares of graphs in terms of their clique number. For a graph $G$, denote the maximum degree of a graph $G$ by $\Delta(G)$. For every vertex $v$ of $G$, the set of all neighbors of $v$ together with $v$ forms a clique in $G^2$, so $\Delta(G)+1\leq \omega(G^2)$ and consequently, $\chi(G^2)\leq \Delta(G^2)+1\leq \Delta^2(G)+1\leq (\omega(G^2)-1)^2+1$. Chakraborty~\etal in~\cite{https://doi.org/10.1002/jgt.70014} observed that the squares of bipartite graphs of girth strictly greater than 6 have a sub-quadratic $\chi$-binding function. In addition, they observed that the question of whether the class of squares of bipartite graphs have a sub-quadratic $\chi$-binding function was still open. We answer this question in the negative. We prove a quadratic lower bound using a class of graphs obtained from the incidence graph of a family of projective planes by performing a combination of deterministic and randomized edit operations on it (see Section~\ref{sec: $r=2$}).

Further, we generalize this result to the $\chi$-binding function of higher powers of bipartite graphs. Our main result shows that the quadratic $\chi$-binding bound is asymptotically best possible for every fixed power of bipartite graphs.

\begin{theorem}\label{thm:main}
For every fixed integer $r\geq 2$, $f_r(t)=\Theta_r(t^2)$,
where
\[
        f_r(t)=\sup\{\chi(G^r):G\text{ is bipartite and }\omega(G^r)\leq t\}.
\]

\end{theorem}

The case $r=1$ is, of course, different: if $G$ is bipartite then $\chi(G)\leq 2$.  Thus theorem~\ref{thm:main} gives the complete asymptotic answer for all nontrivial graph powers of bipartite graphs. 

\begin{remark}
The notion of $\chi$-boundedness is often considered for hereditary graph classes; equivalently, one requires the inequality $\chi(H)\leq f(\omega(H))$ for every induced subgraph $H$ of every graph in the class. The classes considered here need not be hereditary, and hence throughout the paper we use $\chi$-binding function in the pointwise sense above, following Chakraborty et al.~\cite{https://doi.org/10.1002/jgt.70014}.
\end{remark}

\section{Preliminaries}
We first describe a few standard notations used in the paper throughout. All graphs in this paper are finite, simple, and undirected. For a graph $G$, we write $V(G)$ and $E(G)$ for its vertex set and edge set. An independent set is a set of pairwise nonadjacent vertices, and $\alpha(G)$ denotes the maximum size of an independent set. A graph $G$ is bipartite if its vertex set can be partitioned into two independent sets $A$ and $B$; we write $G=(A,B,E)$ when this bipartition is specified.~The girth of $G$, denoted $\girth(G)$, is the length of a shortest cycle in $G$, with $\girth(G)=\infty$ if $G$ is acyclic. If $X\subseteq V(G)$, then $G[X]$ denotes the subgraph induced by $X$. The open and closed neighborhoods of $v$ are denoted by $N_G(v)$ and $N_G[v]=N_G(v)\cup\{v\}$, respectively.  The maximum degree of $G$ is $\Delta(G)$. 

The distance between two vertices $x,y\in V(G)$ is denoted by $\dist_G(x,y)$ and is the length of a shortest $x$-$y$ path in $G$; if $x$ and $y$ lie in different connected components, their distance is infinite.  For $s\geq 0$ and $v\in V(G)$, the ball of radius $s$ around $v$ is $B_G(v,s)=\{x\in V(G):\dist_G(v,x)\leq s\}$.

\begin{definition}[Graph power]\label{def:gpow}
Let $G$ be a graph and let $r\geq 1$ be an integer.  The $r$th power of $G$, denoted by $G^r$, is the graph with vertex set $V(G)$ in which two distinct vertices $x,y$ are adjacent if and only if      $1\leq \dist_G(x,y)\leq r$.
Thus $G^1=G$, and $G^2$ is the square of $G$.
\end{definition}

%\begin{definition}[Line graph]
%The line graph $L(H)$ of a graph $H$ is the graph whose vertices are the edges of $H$.  Two vertices of $L(H)$ are adjacent precisely when the corresponding edges of $H$ share an endpoint.
%\end{definition}

We next define the standard Desarguesian projective plane of order $q$ which we will use for our constructions.

\begin{definition}[The projective plane $\PG(2,q)$]
Let $q$ be a prime power and let $\FF_q$ be the finite field with $q$ elements.  The projective plane $\PG(2,q)$ has:
\begin{itemize}
    \item points equal to the one-dimensional subspaces of $\FF_q^3$;
    \item lines equal to the two-dimensional subspaces of $\FF_q^3$;
    \item incidence given by containment.
\end{itemize}
We write its point set as $\PP$ and its line set as $\LL$.
\end{definition}

Next we state and prove some standard properties of $\PG(2,q)$. We give the proofs for completeness.

\begin{lemma}\label{lem:pgfacts}
The following holds in $\PG(2,q)$:
\begin{enumerate}[(\roman*)]
    \item $|\PP|=|\LL|=q^2+q+1$;
    \item each line contains exactly $q+1$ points;
    \item each point lies on exactly $q+1$ lines;
    \item two distinct points lie on a unique common line;
    \item two distinct lines meet in a unique common point.
\end{enumerate}
\end{lemma}

\begin{proof}
The nonzero vectors of $\FF_q^3$ fall into one-dimensional subspaces, each containing $q-1$ nonzero vectors.  Hence the number of points is $      \frac{q^3-1}{q-1}=q^2+q+1$.

By vector-space duality the number of two-dimensional subspaces, and therefore the number of lines, is the same.

A two-dimensional vector space over $\FF_q$ contains $\frac{q^2-1}{q-1}=q+1$
one-dimensional subspaces, so every line contains $q+1$ points.  The dual statement gives that every point lies on $q+1$ lines.

Two distinct one-dimensional subspaces span a unique two-dimensional subspace, which is the unique line through the corresponding points.  Two distinct two-dimensional subspaces of a three-dimensional vector space intersect in a unique one-dimensional subspace, which is the unique point common to the corresponding lines. 
\end{proof}

\section{Quadratic upper bound for powers of bipartite graphs}\label{sec:upper bound}

In order to prove \cref{thm:main}, we first establish the estimate $f_r(t)=O_r(t^2)$. For even powers, the stronger statement below holds for arbitrary graphs; bipartiteness is needed only for odd powers. The proof is fairly elementary but we give it for completeness.

\begin{lemma}\label{lem:upper-even}
Let $G$ be any graph, let $r=2s$ be an even number with $s\geq 1$, and put $d=\omega(G^r)$.  Then $\chi(G^r)\leq d^2$.
\end{lemma}

\begin{proof}
For every vertex $v$, the ball $B_G(v,s)$ is a clique in $G^r$: any two vertices in $B_G(v,s)$ are at distance at most $2s=r$. Hence    $|B_G(v,s)|\leq d$ for every $v\in V(G)$.
Furthermore,
\[
        B_G(v,2s)\subseteq \bigcup_{x\in B_G(v,s)}B_G(x,s).
\]
Indeed, let $u\in B_G(v,2s)$.  If $\dist_G(u,v)\leq s$, take $x=u$.  Otherwise, take $x$ to be the vertex at distance $s$ from $v$ on a shortest $v$-$u$ path.  Then $x\in B_G(v,s)$ and $u\in B_G(x,s)$.

Therefore $|B_G(v,r)|=|B_G(v,2s)|\leq d^2$ for every $v\in V(G)$. The closed neighborhood of $v$ in $G^r$ is precisely $B_G(v,r)$, so $\Delta(G^r)\leq d^2-1$. Hence, $\chi(G^r)\leq d^2$.
\end{proof}

\begin{lemma}\label{lem:upper-odd}
Let $G=(A,B,E)$ be bipartite, let $r=2s+1$ be an odd number with $s\geq 1$, and put $d=\omega(G^r)$. Then $\chi(G^r)\leq 2d^2$.
\end{lemma}

\begin{proof}
For every vertex $v$, the ball $B_G(v,s)$ is a clique in $G^r$, because any two vertices in it are at distance at most $2s<r$. Thus $|B_G(v,s)|\leq d$ for every $v\in V(G)$.

Inside one bipartition class, all finite distances are even.  Hence two vertices in $A$ are adjacent in $G^r[A]$ only if their distance in $G$ is at most $2s$.  For $v\in A$, the closed neighborhood of $v$ inside $G^r[A]$ is therefore contained in $B_G(v,2s)$.  By the same argument used in \cref{lem:upper-even}, $|B_G(v,2s)|\leq d^2$.
Thus $\Delta(G^r[A])\leq d^2-1$, so $G^r[A]$ is $d^2$-colorable.  The same argument applies to $G^r[B]$.  Hence by using different sets of colors for $A$ and $B$ gives a proper coloring of $G^r$ with at most $2d^2$ colors.
\end{proof}

\section{Lower bound for squares of bipartite graphs}\label{sec: $r=2$}

This section contains the main probabilistic construction of the paper. We prove the lower-bound direction of \cref{thm:main} for $r=2$ by constructing, for every sufficiently large prime power $q$, a bipartite graph $\Gamma_q$ such that $\chi(\Gamma_q^2)=\Omega(q^2)$ and $\omega(\Gamma_q^2)=O(q)$, and hence $\chi(\Gamma_q^2)=\Omega\bigl(\omega(\Gamma_q^2)^2\bigr)$. This construction also forms the basis of the lower bounds for all higher powers that will be discussed in Section~\ref{sec: higher powers lower bound}.

We begin with the \emph{incidence graph} $G_q$ of the projective plane $\PG(2,q)$. Its two partite sets are the point set $\PP$ and the line set $\LL$, and a point is adjacent to a line precisely when it lies on that line. Alon and Mohar~\cite{alon2002chromatic} observed that $\chi(G_q^2)=q^2+q+1$.
However, $\omega(G_q^2)=\chi(G_q^2)$, so the large chromatic number is entirely accounted for by a clique of the same order. Our construction modifies $G_q$ by randomly splitting every line-vertex into two signed copies.  We will see that such a construction allows us to keep the clique number low while still maintaining a relatively high chromatic number, in particular, this reduces the clique number of the square from order $q^2$ to order $q$, while preserving a point-side induced subgraph whose chromatic number has order $q^2$.

We now describe the construction described above more precisely. Fix a prime power $q$, and let $\PG(2,q)$ be the projective plane of order $q$.  Let $N=q^2+q+1$ be the number of points, and set       $\rho=q^{-1/3}$.
For every incidence $(p,\ell)$ of a point $p$ and a line $\ell$ in $\PG(2,q)$, choose an independent Bernoulli random variable    $\xi(p,\ell)\in\{0,1\}$ with $\Pr[\xi(p,\ell)=1]=\rho$ and        $\Pr[\xi(p,\ell)=0]=1-\rho$. Below, we present the construction of the \emph{split projective-plane graph}, built from the incidence graph of $PG(2,q)$ by combining a randomized procedure with a suitable vertex-splitting operation.
 
%by randomly splitting every line into two \textcolor{magenta}{vertices that we call \textit{signed line-vertices}.} 

\emph{Construction:}\label{def:sppg}
The graph $\Gamma_q=(A_q,B_q,E_q)$ is the following bipartite graph.  The part $A_q$ is the set of points of $\PG(2,q)$ and these vertices are called point-vertices. The other part is
\[
        B_q=\{b_{\ell,i}:\ell\text{ is a line of }\PG(2,q),\ i\in\{0,1\}\}.
\]
The vertex $b_{\ell,i}$ is called the signed line-vertex corresponding to the line $\ell$ and sign $i$.  A point $p\in A_q$ is adjacent to $b_{\ell,i}$ exactly when $p\in\ell$ and $\xi(p,\ell)=i$.

Thus a point is joined to exactly one of the two signed copies of each line through it. Intuitively, this construction is designed so that the square is highly chromatic on the point side but has only linear-size cliques. Below we note some basic properites of the construction.

\begin{lemma}\label{lem:basic-split}
For every outcome of the random choices, $\Gamma_q$ is bipartite, is $C_4$-free, and satisfies $\Delta(\Gamma_q)\leq q+1$.
\end{lemma}

\begin{proof}
The graph is bipartite by construction. We know that every point of $\PG(2,q)$ lies on exactly $q+1$ lines. Now observe that if a point $p\in \ell$, then $p$ is adjacent to exactly one corresponding signed line-vertex, so every point-vertex has degree $q+1$.  Every signed line-vertex has degree at most $q+1$.  Hence $\Delta(\Gamma_q)\leq q+1$.

To prove $C_4$-freeness, it is enough to show that two vertices in $A_q$ have at most one common neighbor. We know that two points of the projective plane lie on a unique common line, and therefore they can share at most one signed line-vertex as their neighbor. 

\end{proof}

A set of points in a projective plane is called an \emph{arc} if no three of its points are collinear. For distinct points $x$ and $y$, we will denote by $xy$ the unique line determined by $x$ and $y$ in the projective plane. We next establish the probabilistic properties of $\Gamma_q$ needed for the quadratic lower bound. The constants are chosen for convenience and are not optimized.

\begin{lemma}\label{lem:prob-estimates}
With probability tending to $1$ as $q\to\infty$, the graph $\Gamma_q$ satisfies all of the following:
\begin{enumerate}[(\roman*)]
\item $\alpha\bigl(\Gamma_q^2[A_q]\bigr)\leq 13$;
\item $\omega\bigl(\Gamma_q^2[A_q]\bigr)<4q$;
\item $\omega\bigl(\Gamma_q^2[B_q]\bigr)<4q$.
\end{enumerate}
Consequently, for all sufficiently large prime powers $q$, there is an outcome for which all three properties hold simultaneously.
\end{lemma}

\begin{proof}

We first prove~(i). Let \(S\subseteq A_q\) be an independent set in \(\Gamma_q^2[A_q]\). Then no three points of \(S\) are collinear. Indeed, suppose that distinct points \(x,y,z\in S\) lie on a common line \(\ell\). Among the three labels $\xi(x,\ell),\ \xi(y,\ell),\ \xi(z,\ell)\in\{0,1\}$,  two are equal, say \(\xi(x,\ell)=\xi(y,\ell)=i\). By the definition of \(\Gamma_q\), both \(x\) and \(y\) are adjacent to \(b_{\ell,i}\). Hence \(x\) and \(y\) are at distance two in \(\Gamma_q\), and therefore are adjacent in \(\Gamma_q^2[A_q]\), a contradiction. Thus every \(14\)-vertex independent set in \(\Gamma_q^2[A_q]\) is a \(14\)-point arc, that is, a set of \(14\) points with no three collinear.

Fix a \(14\)-point arc \(S\). For each unordered pair \(\{x,y\}\in\binom{S}{2}\), let \(xy\) denote the unique line through \(x\) and \(y\), and define $E_{x,y} := \bigl\{\xi(x,xy)\neq \xi(y,xy)\bigr\}.$ The vertices \(x\) and \(y\) are nonadjacent in \(\Gamma_q^2[A_q]\) if and only if \(E_{x,y}\) occurs. Moreover, $ \Pr(E_{x,y}) = 2\rho(1-\rho) \leq 2\rho.$

Furthermore, the events $ \bigl\{E_{x,y}:\{x,y\}\in\binom{S}{2}\bigr\}$ are mutually independent. Indeed, distinct pairs of points of \(S\) determine distinct lines; otherwise one line would contain at least three points of \(S\). Consequently, the \(2\binom{14}{2}\) incidence random variables appearing in these events are all distinct, and all incidence variables were chosen independently. Therefore \[ \Pr\bigl[S\text{ is independent in }\Gamma_q^2[A_q]\bigr] = \bigl(2\rho(1-\rho)\bigr)^{\binom{14}{2}} \leq (2\rho)^{91}. \] Let \(X\) be the number of \(14\)-vertex independent sets in \(\Gamma_q^2[A_q]\). Since \(|A_q|=N=q^2+q+1\) and \(\rho=q^{-1/3}\), linearity of expectation gives \[ \begin{aligned} \mathbb{E}[X] &\leq \binom{N}{14}(2\rho)^{91} \\ &\leq N^{14}(2q^{-1/3})^{91} \\ &= O\!\left(q^{28-91/3}\right) = O\!\left(q^{-7/3}\right) = o(1). \end{aligned} \] By Markov's inequality, $\Pr(X\geq 1)\leq \mathbb{E}[X]=o(1).$ Hence, with probability tending to \(1\), the graph \(\Gamma_q^2[A_q]\) has no independent set of size \(14\). This finishes the proof of~(i). \\

We next prove (ii). It is enough to show that, with probability tending to $1$, there is no clique of size $4q$ in $\Gamma_q^2[A_q]$, since every clique of size at least $4q$ contains a clique of size exactly $4q$. Fix $S\subseteq A_q$ with $|S|=s=4q$. For a projective line $\ell$, let $s_\ell=|S\cap\ell|$. If $S$ is a clique in $\Gamma_q^2[A_q]$, then for every line $\ell$ with $s_\ell\geq 2$, we have that for all $x,y\in S\cap\ell$, $\xi(x,\ell)=\xi(y,\ell)$.
Indeed, if $x,y\in S\cap\ell$ are adjacent in $\Gamma_q^2[A_q]$, then, since $\ell$ is the unique line containing both $x$ and $y$, they must be adjacent in $\Gamma_q$ to the same signed line-vertex corresponding to $\ell$.

For $m\geq 2$, the probability that $m$ independent $\{0,1\}$-valued Bernoulli random variables with $\Pr[\xi=1]=\rho$ are all equal is $(1-\rho)^m+\rho^m\leq (1-\rho)^m+\rho(1-\rho)^{m-1}=(1-\rho)^{m-1}\leq e^{-\rho(m-1)}$,
where we used $\rho\leq 1/2$ for all sufficiently large $q$. Since the incidence variables corresponding to distinct projective lines are disjoint and hence independent,
\[
\Pr[S\text{ is a clique in }\Gamma_q^2[A_q]]\leq \exp\left(-\rho\sum_{\ell:s_\ell\geq 2}(s_\ell-1)\right).
\]

Since every point lies on exactly $q+1$ lines, we have $\sum_\ell s_\ell=s(q+1)$. Therefore,
\[
\sum_{\ell:s_\ell\geq 2}(s_\ell-1)\geq s(q+1)-N.
\]
Since $s=4q$ and $N=q^2+q+1$, $s(q+1)-N=3q^2+3q-1\geq \frac{sq}{2}$.
Thus, using $\rho=q^{-1/3}$,
\[
\Pr[S\text{ is a clique in }\Gamma_q^2[A_q]]\leq \exp\left(-\frac{s q^{2/3}}{2}\right)=e^{-2q^{5/3}}.
\]
Hence, by the union bound,
\[
\Pr[\omega(\Gamma_q^2[A_q])\geq 4q]\leq \binom{N}{4q}e^{-2q^{5/3}}\leq \exp\left(4q\log\frac{eN}{4q}-2q^{5/3}\right)=o(1),
\]
because $N=q^2+q+1=O(q^2)$, and hence $4q\log\frac{eN}{4q}=O(q\log q)=o(q^{5/3})$.
This finishes the proof of (ii). \\

It remains to prove (iii). As in the proof of (ii), it is enough to show that, with probability tending to $1$, there is no clique of size $4q$ in $\Gamma_q^2[B_q]$. Fix a set $T\subseteq B_q$ with $|T|=s=4q$. If $T$ contains both signed copies $b_{\ell,0}$ and $b_{\ell,1}$ of some projective line $\ell$, then $T$ is not a clique, since these two vertices have disjoint neighborhoods in $\Gamma_q$ and hence are nonadjacent in $\Gamma_q^2[B_q]$. Thus, we may assume that the vertices of $T$ correspond to $s$ distinct projective lines. For every such line $\ell$, let $\sigma(\ell)\in\{0,1\}$ be the sign for which $b_{\ell,\sigma(\ell)}\in T$.

For a point $p$ of the projective plane, let $t_p$ denote the number of underlying lines corresponding to vertices of $T$ that pass through $p$. Suppose that $T$ is a clique in $\Gamma_q^2[B_q]$. If $t_p\geq 2$ and $\ell$ is one of the underlying lines of $T$ through $p$, choose another such line $m\neq \ell$ through $p$. Since $b_{\ell,\sigma(\ell)}$ and $b_{m,\sigma(m)}$ are adjacent in $\Gamma_q^2[B_q]$, they must have a common neighbor in $\Gamma_q$. As two distinct projective lines $\ell$ and $m$ meet in the unique point $p$, this common neighbor must be $p$. Therefore, $\xi(p,\ell)=\sigma(\ell)$ and $\xi(p,m)=\sigma(m)$.

Hence, whenever $t_p\geq 2$, the label $\xi(p,\ell)$ is forced for every underlying line $\ell$ of $T$ passing through $p$.

Let $M=\sum_{p:t_p\geq 2}t_p$ be the number of forced incidences. Now observe that,
\[
M\geq \sum_{p:t_p\geq 2}(t_p-1)
=\sum_{p:t_p\geq 1}(t_p-1)\geq s(q+1)-N,
\]
as every one of the $s$ underlying lines contains exactly $q+1$ points.

Since $s=4q$ and $N=q^2+q+1$, $s(q+1)-N=3q^2+3q-1\geq \frac{sq}{2}$.
Thus, $M\geq \frac{sq}{2}$.

For a fixed signed set $T$, each forced incidence requires a prescribed value of the corresponding Bernoulli random variable. Since $\rho\leq 1/2$ for all sufficiently large $q$, the probability of any prescribed value is at most $1-\rho$. Moreover, the forced incidence variables are distinct and independent. Therefore
\[
\Pr[T\text{ is a clique in }\Gamma_q^2[B_q]]
\leq (1-\rho)^M
\leq e^{-\rho M}
\leq \exp\left(-\frac{s q^{2/3}}{2}\right)
=e^{-2q^{5/3}},
\]
where we used $\rho=q^{-1/3}$ and $s=4q$.

There are at most $\binom{N}{4q}2^{4q}$ possible signed sets $T$ with $4q$ distinct underlying lines. Hence, by the union bound,
\[
\Pr[\omega(\Gamma_q^2[B_q])\geq 4q]
\leq \binom{N}{4q}2^{4q}e^{-2q^{5/3}}
\leq \exp\left(4q\log\frac{eN}{4q}+4q\log 2-2q^{5/3}\right)
=o(1),
\]
because $N=q^2+q+1=O(q^2)$, and hence the positive part of the exponent is $O(q\log q)=o(q^{5/3})$. This finishes the proof of (iii).
\end{proof}

Combining Lemma~\ref{lem:prob-estimates} with the basic properties of $\Gamma_q$, we obtain the following existence theorem, which establishes the lower-bound direction of Theorem~\ref{thm:main} for $r=2$.

\begin{theorem}\label{thm:r=2 lower bound}
For all sufficiently large prime powers $q$, there exists a bipartite $C_4$-free graph $\Gamma_q$ with the following properties:
\begin{itemize}
\item $\Delta(\Gamma_q)\leq q+1$
        
\item $\chi(\Gamma_q^2)\geq \frac{q^2+q+1}{13}$,
        
\item $\omega(\Gamma_q^2)\leq 4q$.

\end{itemize}

\end{theorem}

\begin{proof}
Choose an outcome satisfying \cref{lem:prob-estimates}; such an outcome exists for all sufficiently large prime powers $q$.  Since $|A_q|=q^2+q+1$ and $\alpha(\Gamma_q^2[A_q])\leq 13$,
\[
        \chi(\Gamma_q^2)
        \geq
        \chi(\Gamma_q^2[A_q])
        \geq
        \frac{|A_q|}{\alpha(\Gamma_q^2[A_q])}
        \geq
        \frac{q^2+q+1}{13}.
\]

We next bound $\omega(\Gamma_q^2)$.  A clique contained in $A_q$ has size less than $4q$ by \cref{lem:prob-estimates}(ii), and a clique contained in $B_q$ has size less than $4q$ by \cref{lem:prob-estimates}(iii).  Finally, consider a clique meeting both $A_q$ and $B_q$.  In a bipartite graph, vertices in opposite parts are adjacent in the square only when they are adjacent in the original graph.  Since $\Gamma_q$ is $C_4$-free, such a mixed clique cannot contain at least two vertices from $A_q$ and at least two vertices from $B_q$: otherwise these four vertices would span a $K_{2,2}$ in $\Gamma_q$ which contradicts the $C_4$-free assertion in Lemma~\ref{lem:basic-split}. Hence a mixed clique has size at most one plus the maximum degree of $\Gamma_q$, and therefore at most $q+2\leq 4q$. It follows that $\omega(\Gamma_q^2)\leq 4q$.
\end{proof}

\begin{remark}
Chakraborty et al.~\cite{https://doi.org/10.1002/jgt.70014} proved that squares of graphs of girth greater than $6$ admit a subquadratic $\chi$-binding function. Consequently, for all sufficiently large $q$, every graph satisfying the properties of Theorem~\ref{thm:r=2 lower bound} must have girth exactly $6$. Indeed, Theorem~\ref{thm:r=2 lower bound} gives $\chi(\Gamma_q^2)\geq \frac{q^2}{13}\geq \frac{1}{208}\omega(\Gamma_q^2)^2$. Moreover, every point-vertex $p\in A_q$ has degree exactly $q+1$ in $\Gamma_q$. Since the closed neighborhood $N_{\Gamma_q}[p]$ forms a clique in $\Gamma_q^2$, we have $\omega(\Gamma_q^2)\geq |N_{\Gamma_q}[p]|=q+2\to\infty$.
Hence $\Gamma_q$ cannot have girth greater than $6$. Since $\Gamma_q$ is bipartite and $C_4$-free, its girth is therefore exactly $6$.

\end{remark}

%\section{\textcolor{magenta}{Proof of theorem 1.1}}

\section{Lower bound for higher powers of bipartite graphs}\label{sec: higher powers lower bound}

Recall that in the preceding section we proved that for every sufficiently large prime power $q$, there exists a $C_4$-free bipartite graph $\Gamma_q=(A_q,B_q,E_q)$ such that
\begin{itemize}
    \item $\Delta(\Gamma_q)\leq q+1$
    \item $\omega(\Gamma_q^2)\leq 4q$
    \item $\chi(\Gamma_q^2[A_q])\geq \frac{q^2+q+1}{13}$.
\end{itemize}
We now use these graphs to obtain quadratic lower bound for the chromatic number in terms of the clique number for every higher power. For $r\geq 3$, we derive the lower bound for the $r^{th}$ power from the square construction by uniformly subdividing every edge of the square example. More precisely, if $L=\lfloor r/2\rfloor$, then $r\in\{2L,2L+1\}$, and distances between the original vertices are multiplied by $L$. Consequently, the subgraph induced by the original vertices in the $r^{th}$ power of the $L$-subdivision is exactly the square of the original graph, so its high chromatic number is preserved. Controlling the clique number is not automatic, since a clique in the $r^{th}$ power may contain internal vertices lying on the subdivided paths corresponding to many distinct edges of the original graph. The proximity of these internal vertices forces the corresponding original edges to be pairwise incident or to have adjacent endpoints, and the $C_4$-free bipartite structure bounds the size of such a family linearly in the maximum degree. The next lemma makes this argument precise.

%For the cube, bipartiteness implies that distances between vertices in the same part are even, and hence $\Gamma_q^3[A_q]=\Gamma_q^2[A_q]$, so the high chromatic number from the square case is preserved.

%For a graph $G$ and an integer $r\geq 1$, recall that the $r$th power $G^r$ has vertex set $V(G)$, with two distinct vertices adjacent whenever their distance in $G$ is at most $r$. 

\begin{definition}[$L$-subdivision]
For an integer $L\geq 1$, the $L$-subdivision $S_L(H)$ is obtained from a graph $H$ by replacing each edge $uv$ with a path of length $L$ joining $u$ and $v$, with the subdivided paths internally disjoint. Thus $S_1(H)=H$.
\end{definition}

\begin{lemma}\label{lem:subdivision-transfer}
Let $H$ be a $C_4$-free bipartite graph of maximum degree $\Delta$, let $L\geq 2$, let $r\in\{2L,2L+1\}$, and put $G=S_L(H)$. Then the following holds:
\begin{itemize}
    \item $G$ is bipartite,
\item $G^r[V(H)]=H^2$

\item $\omega(G^r)\leq \omega(H^2)+4(L-1)\Delta$.
\end{itemize}
\end{lemma}

\begin{proof}
It is easy to see that $G$ is 2-colorable as every edge of $H$ is replaced by a path of the same length $L$, and the original graph $H$ is bipartite, thus the graph $G$ is bipartite.
%Indeed, if $L$ is odd, retain the two colors of the original bipartition and alternate colors along every subdivided path; if $L$ is even, give all original vertices one color and again alternate along every subdivided path. Since the subdivided paths are internally disjoint, these colorings are well defined.

For original vertices $u,v\in V(H)$, every $u$--$v$ path of length $m$ in $H$ lifts to a path of length $Lm$ in $G$. Conversely, every internal vertex of $G$ has degree two and belongs to a unique subdivided path, so a shortest path in $G$ between original vertices traverses each subdivided path that it uses from one endpoint to the other. Contracting those subdivided paths gives a path in $H$. Hence $\dist_G(u,v)=L\dist_H(u,v)$ for all $u,v\in V(H)$.

Since $2L\leq r\leq 2L+1<3L$, the above equality implies that $\dist_G(u,v)\leq r\Longleftrightarrow \dist_H(u,v)\leq 2$, which proves $G^r[V(H)]=H^2$.

Let $C$ be a clique in $G^r$ and put $C_0=C\cap V(H)$. By the equality just proved, $C_0$ is a clique in $H^2$, so $|C_0|\leq \omega(H^2)$. For each edge $e\in E(H)$, let $P_e$ be its subdivided path, and let $\mathcal F$ be the set of edges $e$ for which $C$ contains an internal vertex of $P_e$.

We first show that any two distinct edges in $\mathcal F$ either share an endpoint or have adjacent endpoints in $H$. Suppose otherwise, and choose internal vertices $x\in C\cap V(P_e)$ and $y\in C\cap V(P_f)$. Any $x$--$y$ path must first reach an endpoint $u$ of $e$ and, before reaching $y$, pass through an endpoint $v$ of $f$. The endpoint sets of $e$ and $f$ have distance at least two in $H$, so $\dist_G(u,v)\geq 2L$. Since $x$ and $y$ are internal vertices, $\dist_G(x,u),\dist_G(v,y)\geq 1$. Hence $\dist_G(x,y)\geq 1+2L+1=2L+2>r$, contradicting that $C$ is a clique in $G^r$.

It remains to bound $|\mathcal F|$. If $\mathcal F=\varnothing$, the desired bound is immediate. Otherwise, let $(X,Y)$ be the bipartition of $H$, choose $ab\in\mathcal F$ with $a\in X$ and $b\in Y$, and let $\mathcal F_0$ be the edges of $\mathcal F$ incident with $a$ or $b$. Then $|\mathcal F_0|\leq d_H(a)+d_H(b)-1\leq 2\Delta-1$. Write every edge of $\mathcal F\setminus\mathcal F_0$ as $xy$ with $x\in X$ and $y\in Y$. By the preceding paragraph, either $xb\in E(H)$ or $ay\in E(H)$; both cannot hold, since otherwise $a,y,x,b,a$ is a $4$-cycle. Let $\mathcal F_X$ consist of the edges with $xb\in E(H)$ and $\mathcal F_Y$ of those with $ay\in E(H)$. Thus $\mathcal F_X$ and $\mathcal F_Y$ partition $\mathcal F\setminus\mathcal F_0$.

All edges in $\mathcal F_X$ have the same endpoint in $X$. Indeed, suppose $xy,x'y'\in\mathcal F_X$ with $x\neq x'$. If $y=y'$, then $x,y,x',b,x$ is a $4$-cycle. If $y\neq y'$, then the preceding pairwise property applied to $xy$ and $x'y'$ gives either $xy'\in E(H)$ or $x'y\in E(H)$; in the first case $x,y',x',b,x$ is a $4$-cycle, and in the second case $x',y,x,b,x'$ is a $4$-cycle. Both are impossible. Hence $|\mathcal F_X|\leq \Delta$. By symmetry, $|\mathcal F_Y|\leq \Delta$. Therefore
\[
 |\mathcal F|\leq |\mathcal F_0|+|\mathcal F_X|+|\mathcal F_Y|\leq (2\Delta-1)+\Delta+\Delta\leq 4\Delta.
\]
Each path $P_e$ contains exactly $L-1$ internal vertices, and hence
\[
 |C\setminus C_0|\leq (L-1)|\mathcal F|\leq 4(L-1)\Delta.
\]
Combining this with $|C_0|\leq \omega(H^2)$ proves the claimed clique bound.
\end{proof}

Now we have the necessary estimates and constructions to show the existence of a bipartite graph such that the chromatic number of the $r^{th}$ power of the graph is lower bounded by the square of the clique number of the $r^{th}$ power for all $r\geq 3$. Note that for $r=2$, we have shown its existence in Section~\ref{sec: $r=2$}. The theorem below makes it precise.

\begin{theorem}\label{thm:lower-higher}
For every fixed integer $r\geq 3$ and for every sufficiently large prime power $q$ there exists a bipartite graph $G_{q,r}$ such that
\[
 \chi(G_{q,r}^r)\geq \frac{1}{832\lfloor r/2\rfloor^2}\,\omega(G_{q,r}^r)^2.
\]

\end{theorem}

\begin{proof}
Let $\Gamma_q=(A_q,B_q,E_q)$ be the bipartite graph obtained in Theorem~\ref{thm:r=2 lower bound}.

First let $r=3$ and take $G_{q,3}=\Gamma_q$. Since $\Gamma_q$ is bipartite, distances within either bipartition class are even; hence $\Gamma_q^3[A_q]=\Gamma_q^2[A_q]$ and $\Gamma_q^3[B_q]=\Gamma_q^2[B_q]$. Therefore
\[
 \chi(\Gamma_q^3)\geq \chi(\Gamma_q^2[A_q])\geq \frac{q^2+q+1}{13}.
\]
If $C$ is a clique in $\Gamma_q^3$, then $C\cap A_q$ and $C\cap B_q$ are cliques in $\Gamma_q^2$, so Theorem~\ref{thm:r=2 lower bound} gives $|C\cap A_q|,|C\cap B_q|\leq 4q$. Thus $\omega(\Gamma_q^3)\leq 8q$, and consequently
\[
 \chi(\Gamma_q^3)\geq \frac{q^2}{13}\geq \frac{1}{832}\,\omega(\Gamma_q^3)^2.
\]

Now let $r\geq 4$, put $L=\lfloor r/2\rfloor$, and take $G_{q,r}=S_L(\Gamma_q)$. Then $L\geq 2$ and $r\in\{2L,2L+1\}$. By Lemma~\ref{lem:subdivision-transfer}, $G_{q,r}$ is bipartite and $G_{q,r}^r[V(\Gamma_q)]=\Gamma_q^2$. In particular,
\[
 \chi(G_{q,r}^r)\geq \chi(\Gamma_q^2[A_q])\geq \frac{q^2+q+1}{13}.
\]
The same lemma and Theorem~\ref{thm:r=2 lower bound} give $\omega(G_{q,r}^r)\leq 4q+4(L-1)(q+1)\leq 4L(q+1)\leq 8Lq$.
Therefore
\[
 \chi(G_{q,r}^r)\geq \frac{q^2}{13}\geq \frac{1}{832L^2}\,\omega(G_{q,r}^r)^2,
\]
as required.
\end{proof}

\section{Proof of main theorem}
 For $r\geq 2$ and $t\in\mathbb{N}$, recall that $f_r(t)=\sup\{\chi(G^r):G\text{ is bipartite and }\omega(G^r)\leq t\}$. We now prove the sharp asymptotic estimate for $f_r(t)$ in Theorem~\ref{thm:main} by simply combining the upper and lower bounds obtained in the previous sections.

\begin{proof}[Proof of \cref{thm:main}]
The function $f_r(t)$ is monotone non-decreasing in $t$.  For a fixed value of $r$, from Theorem~\ref{thm:r=2 lower bound} and Theorem~\ref{thm:lower-higher}, we have existence of bipartite graphs $G$ with
$\omega(G^r)\leq C_rq$ and $\chi(G^r)\geq c'_r q^2$
for all sufficiently large primes $q$, where $C_r,c'_r>0$ are constants that only depend on $r$.  

By Bertrand's postulate, for every sufficiently large real $x$ there is a prime $q$ with $x\leq q\leq 2x$.  Given large $t$, choose such a prime with
$\frac{t}{2C_r}\leq q\leq \frac{t}{C_r}$.
Then $\omega(G^r)\leq t$ and $\chi(G^r)\geq c'_r q^2\geq \frac{c'_r}{4C_r^2}t^2$.

Thus $f_r(t)\geq c''_r t^2$ for all sufficiently large $t$.  The upper bound $f_r(t)\leq 2t^2$ for all $t$ follows from Lemma~\ref{lem:upper-even} and Lemma~\ref{lem:upper-odd}. Hence $f_r(t)=\Theta_r(t^2)$.
\end{proof}

%------------------------------------------------------------------------------------------
%\section{Concluding remarks}
%\label{sec:conclusions}
%--------------------------------------------------------------------------
\section{Acknowledgements}
We would like to thank Prof L. Sunil.~Chandran for suggesting us to try this problem. The author Arpan Sadhukhan gratefully acknowledges support from the Anusandhan National Research Foundation (ANRF) through the ANRF-National Post Doctoral Fellowship (N-PDF) for Mathematical Sciences, File No. PDF/2025/001644. 

\section*{Declaration of generative AI}
During the preparation of this work, the authors used ChatGPT 5.5 and 5.6 Sol pro models for exploratory mathematical discussions, language editing,
and assistance in checking arguments. The authors independently
reconstructed and verified every proof, reviewed and edited all
AI-assisted material, and take full responsibility for the content
of the manuscript.

 %\section{Data Availability}
 %Not applicable.
%\section{Declarations}\textbf{Conflict of interest:} The authors declare that they have no conflict of interest to declare that are relevant to the content of this
%article.
%relationships that could have appeared to influence the work reported in this paper.

\bibliography{references}

\newpage

\end{document}